\documentclass[12pt]{amsart}
\usepackage{amsmath, amsthm, amscd, amsfonts}
\usepackage[shortlabels]{enumitem}
\usepackage{color}
\usepackage[all]{xy}

\newcommand{\subscript}[2]{$#1 _ #2$}
\newtheorem{theorem}{Theorem}[section]
\newtheorem{lemma}[theorem]{Lemma}
\newtheorem{proposition}[theorem]{Proposition}
\newtheorem{corollary}[theorem]{Corollary}
\newtheorem{problem}[theorem]{Problem}
\theoremstyle{definition}
\newtheorem{definition}[theorem]{Definition}
\newtheorem{example}[theorem]{Example}

\theoremstyle{remark}

\newtheorem{remark}[theorem]{Remark}
\numberwithin{equation}{section}

\newfont{\kh}{msbm10}

\begin{document}
\title[Atiyah-J\"{a}nich theorem]
{Atiyah-J\"{a}nich theorem for $\sigma$-C*-algebras}
\author{K. Sharifi}
\address{Kamran Sharifi,
\newline Department of Mathematics,
Shahrood University, P.
O. Box 3619995161-316, Shahrood, Iran.}
\email{sharifi.kamran@gmail.com}


\begin{abstract}
K-theory for $ \sigma$-C*-algebras
(countable inverse limits of C*-algebras) has been investigated by
N. C. Phillips [{\it K-Theory} {\bf 3}
(1989), 441--478].
We use his representable K-theory to show that
the space of Fredholm modular operators
with coefficients in an arbitrary unital $ \sigma$-C*-algebra $A$,
represents the functor $X \mapsto {\rm RK}_{0}(C(X, A))$
from the category of countably compactly generated spaces
to the category of abelian groups.
\end{abstract}
\maketitle

\section{introduction}

Atiyah and J\"{a}nich in the 1960s proved that
the space of Fredholm operators
on an infinite dimensional complex Hilbert space equipped with the norm topology represents
the functor $X \mapsto {\rm K}^{0}(X; \mathbb{C})$ from the category of compact
Hausdorff spaces to the category of
abelian groups. Indeed, they presented different methods in \cite{Atiyah1967, JAN} to show that
$$ [X, \mathcal{F}] \to {\rm K}^{0}(X; \mathbb{C})$$
is an isomorphism, where $[X, \mathcal{F}]$ denotes the set of homotopy classes of continuous maps
from the compact Hausdorff space $X$ into the space of Fredholm operators $\mathcal{F}$.
After that,  Atiyah and Singer \cite{AtiyahSinger} proved that the
functor $X \mapsto {\rm K}^{1}(X; \mathbb{C})$ can be represented by a specific path
component of the space of selfadjoint Fredholm operators. Toritsky \cite{Torit, ManTorit} and Mingo \cite{MIN}
generalized the result of Atiyah and J\"{a}nich to the general case
where $A$ is an arbitrary unital C*-algebra. In fact,
Toritsky and Mingo showed, respectively, the abelian groups ${\rm K}^0(X;A)$ and ${\rm K}_0(C(X, A))$
can be realized as the group $[X, \mathcal{F}(H)]$ of homotopy classes of continuous maps
from the compact Hausdorff space $X$ to the space of Fredholm operators
on the standard Hilbert $A$-module $H=l^2(A)$.

In general, inverse limits do not commute with fundamental group and K-theory functors, however,
we show that the Atiyah--J\"{a}nich theorem still holds for
countable inverse limits of C*-algebras (or what are now known as {\it $\sigma$-C*-algebras}).
Indeed, if $X$ is a
countably compactly generated space and $A$ is an arbitrary $ \sigma$-C*-algebra, we show that
the abelian group ${\rm RK}_0(C(X, A))$
can be realized as the group $[X, \mathcal{F}( H )]$ of homotopy
classes of continuous maps from $X$ to the space of Fredholm
operators on $H = l^2(A)$. In particular, this shows that
the Grothendieck group of $A$-vector bundles over $X$ need not be isomorphic to $[X, \mathcal{F}(H)]$.

A pro-C*-algebra
is a complete Hausdorff complex topological $*$-algebra
$ A$, whose topology is determined by its continuous
C*-seminorms in the sense that the net $\{ a_i \}_{i \in I}$
converges to $0$ if and only if the net $\{p( a_i) \}_{i \in I}$
converges to $0$ for every continuous C*-seminorm $p$ on  $
A$. These algebras were first introduced by
Inoue \cite{Inoue1971} as a generalization of
C*-algebras and studied more by Arveson,
Fragoulopoulou,  Phillips and Voiculescu
with different names like
locally C*-algebras and LMC*-algebras in their literature
 \cite{ARV1982, Fragoulopoulou2005, Phillips1988, VOC}.
If the topology of a pro-C*-algebra
is determined by only countably many C*-seminorms then we call it a
{\it $\sigma$-C*-algebra}.
Such algebras arise naturally in the study of coarse Baum--Connes assembly map
\cite{EMMADualizing} and certain aspects of C*-algebras such as tangent algebras of C*-algebras, domain of
closed $*$-derivations on C*-algebras, multipliers of
Pedersen's ideal, noncommutative analogues of classical Lie groups, and K-theory. The
reader is encouraged to study the publications \cite{HartiLukas, Phillips1988, Phillips88lon}
for useful examples and more detailed information.

Phillips has defined a K-theory for $\sigma$-algebras,
which is called representable K-theory and denoted by RK. The RK-theory agrees with
the usual K-theory of C*-algebras if the input is a C*-algebra and important properties
that K-theory satisfies generalize to RK-theory.
He showed that representable K-theory is a representable functor in
the sense that there are $\sigma$-C*-algebras $P$ and $U_{nc}$ such that
${\rm RK}_0(A) \cong [P,A]_1$ and ${\rm RK}_1(A) \cong [U_{nc},A]_1$, where $[A,B]_1$
denotes the set of unital homotopy classes of $*$-homomorphisms from $A$ to $B$,
cf. \cite{Phillips89k, Phillips89ca}. KK-theoy
of $\sigma$-C*-algebras has been studied by Schochet \cite{SchochetIII, SchochetKK} and Weidner \cite{Weidner}
and reconsidered recently
by Emerson and Meyer \cite{EMMADualizing, EMMA} and Mahanta \cite{MahantaKyoto}.
Suppose $H$ is the standard Hilbert module over an arbitrary
unital $\sigma$-C*-algebra $A$. Let $\mathcal{L}(H)$ be the $\sigma$-C*-algebra of bounded adjointable operators
on $H$ and $\mathcal{K}(H)$ be the $\sigma$-C*-algebra of compact operators.
We show that any unitary in $\mathcal{L}(H) / \mathcal{K}(H)$
with trivial class in ${\rm RK}_0(A)$ can be lifted to a unitary in $ \mathcal{L}(H)$. This enables
us to prove our main result as follows. Suppose $X$ is a countably compactly generated space, i.e.,
countable direct limits of compact Hausdorff spaces and $C(X, A)$ is the $\sigma$-C*-algebra of all
continuous functions from $X$ to $A$. Then there is an isomorphism from
$ [X, \mathcal{F}(H)]$ to the abelian group ${\rm RK}_0(C(X,A))$, where
$[X, \mathcal{F}(H)]$ denotes the set of homotopy classes of continuous maps
from $X$ into the space of Fredholm operators on $H$ equipped with
the topology induced by a family of seminorms.

\section{representable k-theory and the milnor ${\rm lim}^{1}$ exact sequence}
In this section we study inverse system of C*-algebras and Hilbert
modules over them and then we recall representable K-theory
and the Milnor $\varprojlim_{n}^{ 1}$ exact sequence.
General definitions and basic facts about inverse systems can be found in the books
\cite{BourbakiTop, JoitaBook1, Fragoulopoulou2005, SchochetPext, Weibel}.


Let $A$ be a pro-C*-algebra and let $S(A)$ denote the set of all continuous C*-seminorms on $A$.
For any $p\in S(A)$,  $\ker(p)= \{a\in A ;~ p(a)=0\}$ is a two-sided
closed $*$-ideal in $A$ and $A_p= A/\ker(p)$ is a C*-algebra.
The set $S(A)$ is directed by declaring $p \leq q$ if and only if $p(a) \leq q(a)$ for all $a\in A$.
Suppose $\kappa_p : A \to A_p$ is the canonical map,
the surjective canonical map $ \pi_{pq}  : A_p \to A_q$ is
defined by $ \pi_{pq}  ( \kappa_p (a))= \kappa_q (a)$ for all $a \in
A$ and $p, q \in S(
A)$ with $p \geq q$. Then $ \{ A_p; \pi_{pq}
\}_{ p,\,q \in S( A),\, p \geq q }$ is an inverse
system of C*-algebras and $ \varprojlim_{p} A_{p}$ is a pro-C*-algebra which can be
identified with $ A$. A pro-C*-algebra $A$ is called a {\it $\sigma$-C*-algebra}
if there is a countable (cofinal) subset $S'$ of $S(A)$
which determines that topology of $A$.
Given any $\sigma$-C*-algebra $A$ one can find a confinal subset of $S(A)$ which corresponds to $ \mathbb{N}$
and explicitly write $A$ as a countable inverse limit of C*-algebras, i.e. $A \cong \varprojlim_{n} A_{n}$.
Furthermore, the connecting $*$-homomorphisms of the inverse system $\{A_n;~ \pi_{n} \}_{n \in \mathbb{N} }$
can be arranged to be surjective without altering the inverse
limit of C*-algebras.
For $\sigma$-C*-algebras $A \cong \varprojlim_{n} A_n$ and $B \cong \varprojlim_{n} B_n$,
$M(A)$ denotes the multiplier algebra of $A$ and $A \otimes B$ denotes the the minimal
tensor product of $A$ and $B$. In particular, we have
$M(A) \cong \varprojlim_{n} M(A_n)$ and $A \otimes B \cong \varprojlim_{n} A_n \otimes B_n$.
For the general theory of
pro-C*-algebra we refer to \cite{Bhatt1991, Fragoulopoulou2005,Phillips1988}.

Hilbert modules over $\sigma$-C*-algebras are the same as over ordinary C*-algebras,
except that inner product, instead of being C*-valued, takes
its values in a $\sigma$-C*-algebra. A Hilbert module over a $\sigma$-C*-algebra $A$
gets its complete topology from the family of seminorms
$\| \cdot \|_{p}= p( \langle \cdot, \cdot \rangle) ^{1/2}$.
Let $E$ be a Hilbert $A$-module and $p_{n}$ be an arbitrary continuous C*-seminorm on $A$.
Then $ N_{p_n}^{E} =\{\xi \in E; \
\| \xi \|_{p_n} =0\}$ is a closed submodule of $E$  and
$E_{n}=E / N_{p_n}^{E}  $ is a Hilbert
$A_{n}$-module with the action $(\xi
+ N_{p_n}^{E}  )\kappa_{n}(a)=\xi a+N_{p_n}^{E}  $ and the inner product
$ \left\langle \xi +N_{p_n}^{E}  , \eta + N_{p_n}^{E}   \right\rangle =\kappa_n
(\left\langle \xi ,\eta \right\rangle ).$ Suppose
$\sigma_{n} : E \to E_{n}$ is the
canonical map, the
surjective canonical map $\varsigma_n
: E_{n+1} \to E_{n}$ is defined by $\varsigma _{n}  (\sigma
_{n+1}  (\xi ))=\sigma _{n}  (\xi )$ for all $\xi \in E$. Then
$\{E_{n}; A_{n};\varsigma _{n}  ,\pi _{n} \}_{ n \in \mathbb{N}}$
is an inverse system of Hilbert C*-modules
and $\varprojlim_{n} E_{n}$ is a Hilbert $A$-module which
can be identified with $E$. The set
$\mathcal{L}(E)$ of all bounded adjointable $
A$-module maps on $E$ becomes a $\sigma$-C*-algebra
with the topology defined by the family of seminorms
$ \|t \|_{n}= \| ( \kappa_n)_{*}(t) \|$, where the surjective morphisms
$ ( \kappa_n)_{*} : \mathcal{L}(E)  \to
\mathcal{L}(E_n)$ are defined by $ ( \kappa_n )_{*}(t)( \sigma_n(\xi) )= \sigma_n( t \xi)$.
The surjective morphisms $ ( \pi_{n})_{*}: \mathcal{L}(E_{n+1}) \to
\mathcal{L}(E_n)$ are defined by $(\pi_{n})_{*}(t_{n+1})( \sigma_{n}( \xi)) = \varsigma
_{n}(t_{n+1}( \sigma_{n+1}( \xi))$. Then $ \{ \mathcal{L}(E_n); \ (\pi_{n})_{*} \}_{ n \in \mathbb{N}}$
is an inverse system of C*-algebras and $\varprojlim_{n}
\mathcal{L}(E_n)$ can be identified by
$\mathcal{L}(E)$. The set  $\mathcal{K}(E)$ of all compact operators on $E$ is
a closed two sided $*$-ideal of $\mathcal{L}(E)$. Consider the restriction of $(\pi_{n})_{*}$
to $ \mathcal{K}(E_n)$, denoted by $(\pi_{n})_{*}$ again,
then the family $ \{ \mathcal{K}(E_n); \ (\pi_{n})_{*} \}_{
n \in \mathbb{N}}$ is an inverse system of C*-algebras and
$\varprojlim_{n} \mathcal{K}(E_n)$ can be identified by $\mathcal{K}(E)$. Moreover,
we have the isomorphism $\mathcal{L}(E) \cong M( \mathcal{K}(E))$. The $\sigma$-C*-algebra
$\mathcal{L}(E) / \mathcal{K}(E)$ is denoted by $ \mathcal{Q}(E)$ which can be identified by
$\varprojlim_{n} \mathcal{Q}(E_n)$. The quotient maps
$ \mathcal{L}(E) \to  \mathcal{L}(E) / \mathcal{K}(E)$ and
$ \mathcal{L}(E_n) \to  \mathcal{L}(E_n) / \mathcal{K}(E_n)$
are denoted by $\rho$ and $\rho_n$, respectively.
Hilbert modules over pro-C*-algebras have
been studied systematically in the book \cite{JoitaBook1}.

Let $A$ be a $\sigma$-C*-algebra and $E$ be Hilbert a $A$-module.
We define the set of {\it Fredholm operators}
and {\it essentially unitary operators} on $E$ by
$$\mathcal{F}(E):=\{ T \in \mathcal{L}(E); \rho(T)~{\rm is~ invertible ~in~}  \mathcal{Q}(E) \},$$
$$\mathcal{KC}(E):=\{ T \in \mathcal{L}(E); \rho(T)~{\rm is~a ~ unitary ~in~}  \mathcal{Q}(E) \},$$
respectively. Indeed, $t \in \mathcal{L}(H)$ is Fredholm if and only if it
has a left generalized inverse as well as a
right generalized inverse mod $\mathcal{K}(E)$  i.e., there are $s_1, s_2 \in \mathcal{L}(E)$ such
that $s_{1}t=1 \ {\rm mod} \ \mathcal{K}(E)$ and $ts_{2}=1 \ {\rm mod} \ \mathcal{K}(E)$. We can equip the
above sets with the topology which is
generated by the operator-seminorms $\| \cdot \|_n$, and consider them as topological spaces. The following
lemma can be deduced from the above definitions and the
isomorphism $\mathcal{L}(E) \to \varprojlim_{n} \mathcal{L}(E_n)$.

\begin{lemma}\label{Fred1} Suppose $t \in \mathcal{L}(E)$. Then
$t$ is Fredholm if and only if $ (\kappa_{n})_{*}(t)
\in \mathcal{L}(E_n)$ is Fredholm for every $n \in \mathbb{N}$.
In particular, for every $n$ we have
$(\kappa_{n})_{*} (\mathcal{F}(E))~
\subseteq ~~\mathcal{F}(E_n)$ and $
(\pi_{n})_{*} (\mathcal{F}(E_{n+1}))
\subseteq \mathcal{F}(E_n).$ The family
$ \{ \mathcal{F}(E_n); \ (\pi_{n})_{*} \}_{n \in \mathbb{N}}$
is an inverse system of topological
spaces and the map $ \Psi : \mathcal{F}(E) \to
 \varprojlim_{n} \mathcal{F}(E_n)$  given by
$ \Psi(t)= ((\kappa_n)_{*}(t))_n$
is a homeomorphism between Hausdorff locally convex spaces.
The assertions hold if the space of Fredholm operators $\mathcal{F}$ is replaced by
the space of essentially unitary operators $\mathcal{KC}$.
\end{lemma}

We assume $A$ to be a unital $\sigma$-C*-algebra. The standard Hilbert $A$-module
$$l^2(A)= \{ \xi \in \prod_{k=1} ^{ \infty} A :~~~
\sum_{k=1} ^{ \infty} \xi_{k}^{*} \xi_{k} ~{\rm~converges ~in}~ A \},$$
with the inner product $ \langle \xi, \eta \rangle =
 \sum_{k=1} ^{ \infty} \xi_{k}^* \eta_{k}$ will be
denoted by $H$ or $H_A$. If $E$ is a countably generated Hilbert $A$-module, then the
Hilbert $A$-modules $E \oplus H$ and $H$ are isomorphic \cite[Theorem 5.2.7]{JoitaBook1}.

\begin{lemma}\label{Mingo1}{\rm (\cite[Proposition 1.5]{MIN})} Any unitary element in
$\mathcal{Q}(H)$ can be lifted to a partial isometry in $ \mathcal{L}(H)$.
\end{lemma}
\begin{proof} Suppose $u$ is a unitary element in $\mathcal{Q}(H)$ and $\mathcal{Q} \kappa_n:
\mathcal{Q}(H) \to \mathcal{Q}(H_n)$ is the canonical map. Then for each $n$, $u_n=\mathcal{Q} \kappa_n (u)$
is a unitary element in $\mathcal{Q}(H_n)$, and so by \cite[Proposition 1.5]{MIN},
there exists a partial isometry  $v_n \in
\mathcal{L}(H_n)$ such that $ \rho_n(v_n)=u_n$. We define $v: H \to H$ by $v x = ( v_n  x_n)$, for
every $x=(x_n) \in H = \varprojlim_{n} H_n $. Then $v$ is well defined, that is,
$ \varsigma_{n}(  v_{n+1}  x_{n+1} )= (\pi_n)_{*}(v_{n+1})( \sigma_n(x))=v_{n} x_{n}$, for each $n$.
Using \cite[Remark 3.1.3]{JoitaBook1}, the operator $v$ is a partial isometry
in $ \mathcal{L}(H)$ and satisfies $ \rho(v)=u$.
\end{proof}

Let $\mathcal{L}$ be the C*-algebra of bounded operators on
the separable infinite dimensional Hilbert space $l^2(\mathbb{C})$
and $ \mathcal{K}$ be its closed two sided $*$-ideal of compact operators.
Suppose $A$ and $B$ are $\sigma$-C*-algebras
then $A \otimes B$ is an essential ideal
in both algebras $M(A \otimes B)$ and $M(A) \otimes M(B)$, and
so we may regard $M(A) \otimes M(B)$ as a $\sigma$-C*-subalgebra of $M(A \otimes B)$.
In particular, when $A$ is a unital $\sigma$-C*-algebra, $A \otimes \mathcal{L}$ can be
regarded as a $\sigma$-C*-subalgebra of
$\mathcal{L}(H_A) \cong M(A \otimes \mathcal{K})$. The following result follows
from \cite[Lemma 1.10]{MIN} and the facts
$A  \otimes \mathcal{K} \cong \varprojlim_{n}  A_n  \otimes \mathcal{K}$ and
$A \otimes \mathcal{L} \cong \varprojlim_{n} A_n \otimes \mathcal{L}$.

\begin{lemma}\label{Mingo2}
Suppose that $A$ is a unital $\sigma$-C*-algebra and $p \in A  \otimes \mathcal{K}$
is a projection. Then there is an element $w$ in $A \otimes \mathcal{L}$ with $w^*w=1$ and $1-ww^*=p$.
\end{lemma}

Suppose $X$ and $Y$ are topological Hausdorff spaces. Let
$[X,Y]$ denote the set of homotopy classes of continuous maps
from $X$ into $Y$ and let $\pi_n(Y)$ denote the $n$-th homotopy group of $Y$ (i.e.
the set of homotopy classes of continuous maps from the $n$-sphere $S^n$ into $Y$). In particular,
$\pi_0(Y)$ is the set of path components of $Y$.
Let $C(X,Y)$ denote
the set of all continuous functions from $X$ to $Y$ with the compact-open topology.
The space $C(X, \mathbb{C})$ is denoted by $C(X)$. A space
$X= \varinjlim_{n} X_n$
is a {\it countably compactly generated space} if each $X_n$ is compact and Hausdorff,
i.e., $X$ is a countable direct limit of compact Hausdorff spaces in the category of topological
spaces and continuous maps. This is not the same as being $\sigma$-compact and compactly generated.
A countably compactly generated space is paracompact if it is regular \cite{FrankSmith}.

If $A= \varprojlim_{n} A_n$ is a unital $\sigma$-C*-algebra and $X= \varinjlim_{n} X_n$
is a countably compactly generated space, then $C(X, A)$ is the algebra
of all continuous functions from $X$ to $A$ is a $\sigma$-C*-algebra with the topology
determined by the family of C*-seminorms $\| f \|_{X_n,~ \iota_n,~ p_m}=
{\rm sup}_{x \in X_n} \, p_m( f \circ \iota_n (x))$ for canonical morphisms
$\iota_n : X_n \to X$ and $p_m \in S(A)$. This topology is equivalent
with the topology of uniform convergence on each $X_n$ in each continuous
C*-seminorm $p_m$. The obvious map from
$C(X) \otimes A$ to $C(X,A)$ is an isomorphism \cite[Proposition 3.4]{Phillips1988}.
In addition, we have the natural isomorphisms
$C(X,A) \cong \varprojlim_{n} C(X,A_n)$ and  $C(X,A) \cong \varprojlim_{n} C(X_n,A)$.
For $X$ locally compact,
$C_0(X)$ is as usual the C*-algebra of continuous complex-valued functions on $X$
which vanish at infinity.

One should be aware that the inverse limit can also be
constructed inside the category of C*-algebras; however, the two results will not agree. For
instance, if $X= \varinjlim_{n} X_n$
is a countably compactly generated space then $C(X)= \varprojlim_{n} C(X_n)$
inside the category of
topological $*$-algebras, while that inside the category of C*-algebras is
$C_b(X) = C( \beta X )$, where $\beta X$ is the
Stone-\u{C}ech compactification of $X$, see e.g. \cite{Phillips88lon} for more details.

The group of unitary elements of a unital $\sigma$-C*-algebra $A$
will be denoted by $U A$ and the path component of the identity in $U A$
will be denoted by $U_0 A$.
In contrast to the case of a C*-algebra, the subgroup $U_0 A$ not
only need be open in $U A$,
it need not even be closed \cite[Example 3.7]{Phillips89k}.

The appropriate K-theory for $\sigma$-C*-algebras is representable K-theory, which is denoted by RK,
is an extension of the topological K-theory of C*-algebras.

\begin{definition}  {\rm (\cite[Definition 2.1]{Phillips89k})}
Let $A$ be a unital $\sigma$-C*-algebra and $H=l^2(A)$. Then we define
${\rm RK}_0( A)= U \mathcal{Q}(H) / U_0 \mathcal{Q}(H)$ and ${\rm RK}_i(A)=
{\rm RK}_0(S^{i}A)$, where $SA=A \otimes C_0( \mathbb{R})$ is the suspension of $A$ and
$S^{i}A=A \otimes C_0( \mathbb{R}^{i})$ is the $i$-th suspension of $A$.
\end{definition}

Each ${\rm RK}_i$ is a well defined homotopy invariant functor from
the category of $\sigma$-C*-algebras to the category of abelian groups.
If $A$ is a unital C*-algebra, then ${\rm RK}_i(A)$ is naturally isomorphic to
the usual K-theory ${\rm K}_i( A)$. One should be aware
that ${\rm RK}_0$ is quite different with the group made from finitely
generated projective modules over $A$.
Stability, Bott periodicity and six term (long) exact sequence for RK-theory and
more detailed information about the above facts can be found in \cite{Phillips89k}.
The abelian group ${\rm RK}_0(A)$ is isomorphic with the set
of homotopy classes in the set of projections $p$ in the unitization $\mathcal{K}(H \oplus H)^{+}$
such that $p - \left[ \begin{smallmatrix} 1 & 0 \\ 0 &
0   \end{smallmatrix} \right] \in \mathcal{K}(H \oplus H)$. This enables
Phillips to construct $\sigma$-C*-algebras $P$ and $U_{nc}$ equipped with the
appropriate analog of an $H$-group structure,
such that the natural isomorphisms of abelian groups
${\rm RK}_0(A) \cong [P,A]_{1}$ and ${\rm RK}_{1}(A) \cong [U_{nc}, A]_{1}$
are fulfilled  \cite{Phillips89ca, Phillips92Aus}.
Here $[A,B]_{1}$ denotes the set of unital homotopy classes of $*$-homomorphism
from $A$ to $B$.

The direct limit of C*-algebras is always a C*-algebra which implies the continuity
of the K-functor with respect to the inductive limit, see e.g., \cite{WEG}.
In contrast to the case of
direct limits of C*-algebras, the behavior of K-theory with respect to inverse
limits is more complicated.
\begin{example}
Kawamura \cite{KAW} constructed an inverse system of Cuntz algebras $\{ \mathcal{O}_n:
2 \leq n < \infty \}$ with non-surjective connecting maps $ \mathcal{O}_{n+1}
\to \mathcal{O}_{n}$ whose inverse limit is $*$-isomorphic onto $\mathcal{O}_{\infty}$.
The inverse system obtains the following fact:
$$ {\rm K}_0( \varprojlim_{n} \mathcal{O}_{n}) \cong {\rm K}_0(\mathcal{O}_{\infty})
 \cong \mathbb{Z} \not\cong  \hat{ \mathbb{Z}}  =
\varprojlim_{n}  \mathbb{Z} / n \mathbb{Z} \cong \varprojlim_{n} {\rm K}_0( \mathcal{O}_{n} ). $$
\end{example}

Suppose $K_1 \subset K_2 \subset ...$ is a sequence of CW-complexes and ${\rm H}^*$
is an additive cohomology theory.  Milnor  \cite{Milnor} first showed that
$$ 0 \to \varprojlim_{n}\, ^{1} \ {\rm H}^*(K_n) \to {\rm H}^*( \cup K_n) \to
  \varprojlim_{n} {\rm H}^* (K_n) \to 0  $$ is an exact sequence where $\varprojlim_{n}^{ 1}$
is the derived functor of the inverse limit.
We refer the reader to the monographs \cite{Switzer,Weibel} for details concerning
derived functors of the inverse limit.
The above exact sequence, which is known as
{\it Milnor $\varprojlim_{}^{ 1}$ sequence}, has been studied by
 Cohen \cite{Cohen}, Phillips \cite{Phillips89k, Phillips91k} and
 Schochet \cite{SchochetUCT, SchochetPext}
and reconsidered recently by Guentner and Yu \cite{GueYu} in
homotopy and K-theory.

\begin{theorem} \label{philllimit} {\rm (\cite[Theorem 3.2]{Phillips89k})}
Let $\{ A_n \}_{n \in \mathbb{N}}$ be an
inverse system of $\sigma$-C*-algebras with surjective maps $A_{n+1} \to A_{n}$,
which can always be arranged,
then we have the following Milnor $\varprojlim_{n}^{ 1}$ sequence
\begin{equation}\label{ati1}
  0 \to \varprojlim_{n}\, ^{1} \ {\rm RK}_{1-i}(A_n) \to {\rm RK}_i( \varprojlim_{n} A_n) \to
  \varprojlim_{n} {\rm RK}_i(A_n) \to 0.
\end{equation}
\end{theorem}
Gray \cite{Gray} showed that $\varprojlim_{}^{ 1}$ of countable groups is either
$0$ or is uncountable. Specially,  the $\varprojlim_{}^{1}$-term in the exact sequence (\ref{ati1}),
even in nice cases, may be non-zero.
To see this, just look at the sequence
$$X_1 { \overset{f_1} \longrightarrow} \cdot\cdot\cdot~  X_n  { \overset{f_n} \longrightarrow} X_{n+1} \longrightarrow \cdot\cdot\cdot$$
where each $X_n$ is homotopy equivalent to the sphere $S^1$ and
each $f_n: X_n \to X_{n+1}$ is the standard map of degree $p$. Here $p$ is a prime number
and $X_{n+1}$ arises from rotating the sphere $S^1$ around itself
$p^{n}$ times. Then
$ \varprojlim_{n} ^{1} \, {\rm K}_0( C(X_n))= 
 \mathbb{Z}_p / \mathbb{Z}$ which is an uncountable group. 
In view of the exact sequence (\ref{ati1}), the Mittag-Leffler condition and the fact that
finite abelian groups satisfy the descending chain condition, we have the
following result.

\begin{proposition} \label{MittLeff} Suppose  $\{ A_n \}_{n \in \mathbb{N}}$ is an
inverse system of C*-algebras for which the morphisms ${\rm K}_i(A_{n+1}) \to {\rm K}_i(A_{n})$
are surjective, or the abelian groups ${\rm K}_i(A_{n})$ are finite.
Then we have the group isomorphism
$ {\rm RK}_i( \, \varprojlim_{n} A_n) \cong  \varprojlim_{n} {\rm K}_i(A_n)$.
\end{proposition}

\section{atiyah-j\"{a}nich theorem}
In this section we reformulate some results of Mingo \cite{MIN}
to prove the Atiyah-J\"{a}nich theorem in the framework of
$\sigma$-C*-algebras and obtain a Milnor exact sequence of
homotopy classes for the space of Fredholm operators.

If $H$ is the standard
Hilbert module over an arbitrary unital C*-algebra $A$,
Mingo \cite[Proposition 1.11]{MIN} showed that
any invertible element in $ \mathcal{Q}(H)$
with index zero can be lifted to an invertible operator in $ \mathcal{L}(H)$.
The following argument shows that his result is actually true for unitaries in place
of invertible elements. 
\begin{remark} \label{prop1.5} Suppose $H$ is the standard Hilbert
module over an arbitrary unital C*-algebra $A$.
Then any unitary element $ \hat u $ in $ \mathcal{Q}(H)$
with index zero can be lifted to a unitary operator in $ \mathcal{L}(H)$.
To see this suppose $\hat u= \rho (v )$
for a partial isometry $v $ in
$\mathcal{L}(H)$. Then
\begin{equation} \label{ati20}
0= \delta([\hat u]) = \delta ( [ \rho (v )])=[1-v^* v] - [1 - v v^*],
\end{equation}
where $\delta$ is the index map ${\rm K}_{1}( \mathcal{Q}(H )) \to {\rm K}_0( \mathcal{K}(H ))$. We define
$p =1-v^*  v  \in \mathcal{K}(H )$ and $q =1 - v  v^*  \in \mathcal{K}(H )$. Then, by
(\ref{ati20}) there is a compact projection $r $ such that $p  \oplus r $ is equivalent to $q  \oplus r $ and
$r $ is orthogonal to $p $ and $q $. Hence, $ w  \in \mathcal{K}(H )$ exists such that
$w^*  w = p  \oplus r $ and $w  w^* = q  \oplus r $. The operator $v+w$ is close to a unitary operator, which
is actually defined by
$u := v (1-r ) \oplus ( ( q  \oplus v  r )w )$. Since
$ (v  r )^*  (v  r )=r (1-p )r =r $ and $ (v  r ) (v  r )^*=v  r  v^* $, we find
$u^*  u   =  (1 - (p  \oplus r )) \oplus ( w^* (q  \oplus r ) w ) = 1$ and
$u  u^* =1$. We also have $ \rho (u )= \rho (v ) \ \rho (1 -r ) +
\rho  ( q  \oplus v  r ) \ \rho (w )= \rho (v )+0,$ as desired.
\end{remark}

Suppose $H$ is the standard Hilbert module over an
arbitrary unital $\sigma$-C*-algebra $A$. Then the following
assertion follows from Remark \ref{prop1.5},
or alternatively from \cite[Lemma 1.11]{Phillips89k}.

\begin{lemma}\label{lemma1}
Any unitary element in $ \mathcal{Q}(H)$
with trivial class in ${\rm RK}_0(A)$ can be lifted to a unitary in $ \mathcal{L}(H)$.
\end{lemma}


If $u$ is a unitary element in $\mathcal{Q}(H)$, we write $[u]$
for its class in ${\rm RK}_0(A)$.
We equip the set $\mathcal{KC}(H)$ with the topology which is
generated by the operator-seminorm $\| \cdot \|_n$. By $[ \mathcal{KC}(H)]$ we refer to the set of
path components in $\mathcal{KC}(H)$. The equivalence classes are also denoted by $[ \cdot]$ and
if $t,s \in \mathcal{KC}(H)$ we define the product $[st]: = [s] \ [t]$.
The multiplication is well defined and makes the set
$ [\mathcal{KC}(H)]$ into an abelian semigroup just as $ \mathcal{KC}(H)$.

\begin{lemma}\label{ati4}
The set of homotopy classes of $\mathcal{KC}(H)$ with the above operation is
isomorphic to the abelian group ${\rm RK}_0( A)$.
\end{lemma}

\begin{proof}
Suppose $\rho : \mathcal{L}(H) \to
 \mathcal{Q}(H)$ is the quotient map, we define
$ \Gamma : [\mathcal{KC}(H)] \to {\rm RK}_0(A) $,  $[t] \mapsto [ \rho(t) ]$. Then $\Gamma$ is a
homomorphism. To see this, we assume $s$ and $t$ are in $\mathcal{KC}(H)$ and define
the well known path
\begin{equation*} 
w_{ \theta} : = \begin{bmatrix} s & 0 \\ 0 & 1 \end{bmatrix} \,
\begin{bmatrix} {\rm cos} \frac{\pi \theta}{2} & -{\rm sin} \frac{\pi \theta}{2} \\
{\rm sin} \frac{\pi \theta}{2} & {\rm cos} \frac{\pi \theta}{2} \end{bmatrix} \,
\begin{bmatrix} t & 0 \\ 0 & 1 \end{bmatrix} \,
\begin{bmatrix} {\rm cos} \frac{\pi \theta}{2} & {\rm sin} \frac{\pi \theta}{2} \\
-{\rm sin} \frac{\pi \theta}{2} & {\rm cos} \frac{\pi \theta}{2} \end{bmatrix}.
\end{equation*}
Then $\rho (w_{ \theta})$ is in $U \mathcal{Q}(H \oplus H)$ for all $ \theta \in [0,1]$ and so
$ \theta \mapsto w_{ \theta}$ is a continuous path of essentially unitary operators which connects
$ s \oplus t$ to $st \oplus 1$. The result
now follows from the stabilization and the fact that $ \Gamma( [s \oplus t]) = \Gamma([s]) + \Gamma([t])$.

Suppose $[ \rho(t)]=0$ and $t \in \mathcal{KC}(H)$. Then, by Lemma \ref{lemma1},
there is $k \in \mathcal{K}(H)$ such that $u=t+k$ is unitary. So $\theta \mapsto t+ \theta k$, for $\theta \in
[0,1]$, is a path of essentially unitary operators connecting $t$ to $u$.
Utilizing \cite[Lemma 1.9]{Phillips89k}, there is a path of unitary operators connecting $u$ to $1$.
If $[t] \in [\mathcal{KC}(H)]$, the above argument and the fact that
$$\Gamma([t]) + \Gamma([t^*])= \Gamma([t \, t^*])= [ \rho(t \, t^*)]=0,$$
imply $[t] \, [t^*] = [t \, t^*]= [1]$. Thus $[ \mathcal{KC}]$ is an abelian group
and $\Gamma$ is an isomorphism.
\end{proof}

\begin{lemma}\label{Janich}
Suppose that $A$ and $B$ are unital $\sigma$-C*-algebras and $t \in \mathcal{KC}(H_{A \otimes B})$.
Then there exists $z \in \mathcal{KC}(H_{A \otimes B}) \cap A \otimes \mathcal{L}(H_B)$
such that $[t] = [z]$.
\end{lemma}
\begin{proof}
Using Lemma \ref{Mingo1}, there exists a partial isometry $v$ in $\mathcal{L}(H_{A \otimes B})$ such that $\rho(t)=\rho(v)$.
Then $1-v^* v$ and $1-v v^*$ are projections in $\mathcal{K}(H_{ A \otimes B})$ and so, by Lemma \ref{Mingo2},
there are $w_1$ and $ w_2$ in $ A \otimes B
\otimes \mathcal{L} \subset A \otimes M( B \otimes \mathcal{K}) \cong A \otimes \mathcal{L}(H_{B})$ such that
$$w_{1}^* w_1=1, ~ 1-w_1 w_{1}^{*} = 1-v^*v, ~ w_{2}^* w_2=1 ~{\rm and}~ 1-w_2 w_{2}^{*} = 1-v v^*.$$
Let $z=w_1 w_{2}^{*}$ then $z$ is in the subalgebra $A \otimes \mathcal{L}(H_{B})$ of $ \mathcal{L}(H_{A \otimes B})$
and $v^*v=zz^*$ and $vv^*=z^*z$. Since $U  \mathcal{L}(H_{A \otimes B} \oplus H_{A \otimes B})$
is path connected and
$$ u = \begin{bmatrix} v & 1-vv^* \\ -(1-v^*v) &
v^* \end{bmatrix} $$ and
$$\mu = \begin{bmatrix} z^* & 1-z^* z \\ -(1-z z^*) &
z  \end{bmatrix}$$
are unitaries in $ \mathcal{L}(H_{A \otimes B} \oplus H_{A \otimes B})$,
we can find a continuous path
$ \theta \mapsto u_{ \theta}= \left[ \begin{smallmatrix} (u_{ \theta})_{11} & (u_{ \theta})_{12} \\ (u_{ \theta})_{21} &
(u_{ \theta})_{22}   \end{smallmatrix} \right]$ 
in $ U  \mathcal{L}(H_{A \otimes B} \oplus H_{A \otimes B}) $ such that $u_0=u$ and $u_1= \mu$. Therefore,
$  \theta \mapsto (u_{ \theta})_{11}$ is a continuous path from $(u_0)_{11}=v$ to $(u_1)_{11}=z^*$. We have therefore found
an element $z^* \in  A \otimes \mathcal{L}(H_B) \subset \mathcal{L}(H_{A \otimes B})$ with
$1 - zz^*, 1-z^*z \in \mathcal{K}(H_{A \otimes B})$ and $[v]=[z^*]$ in $\mathcal{L}(H_{A \otimes B})$.  Since
$v$ is a compact perturbation of $t$, we have $[t]=[v]=[z^*]$ in $\mathcal{L}(H_{A \otimes B})$.
\end{proof}

Let $X$ be a countably compactly generated space and let
$A$ be a unital $\sigma$-C*-algebra. Let $XA$ denote the $\sigma$-C*-algebra of
continuous functions from $X$ to $A$ and let $XH$ denote the Hilbert $XA$-module
consisting of continuous functions $X$ to $H$. The inner product on $XH$ is given by
$ \langle f,g   \rangle :x \mapsto \langle f(x), g(x) \rangle$ for $f,g \in XH$.
Then $\mathcal{K}(XH) \cong X \mathcal{K}(H)$ is the
canonical isomorphism of $\sigma$-C*-algebras and moreover, $XM( \mathcal{K}(H))$ can be regarded
as a $\sigma$-C*-subalgebra of $M( \mathcal{K}(XH))$.

\begin{lemma}\label{JanichInj}  {\rm (\cite[Theorem 10.2]{WEG})}
Suppose that $A$ and $B$ are unital $\sigma$-C*-algebras then
$$ {\rm RK}_0( M(A \otimes \mathcal{K}) \otimes B) = {\rm RK}_1( M(A \otimes \mathcal{K}) \otimes B) = \{0 \}.$$
In particular, if $X$ is a countably compactly generated space and $H=l^2(A)$, then
the unitary group of  $ X \mathcal{L}(H)$ is path connected.
\end{lemma}
\begin{proof} Suppose $A \cong \varprojlim_{n} A_n$ and $B \cong \varprojlim_{n} B_n$.
Then for all natural numbers $m$ and $n$ we have
$$ {\rm K}_0( M(A_n \otimes \mathcal{K}) \otimes B_m) = {\rm K}_1( M(A_n \otimes \mathcal{K}) \otimes B_m) = \{0 \}.$$
The first part now follows from the Milnor $\varprojlim_{n}^{1}$-sequence (\ref{ati1}) and the second part
follows from ${\rm RK}_0( C(X) \otimes M(A \otimes \mathcal{K})) = \{ 0 \}$ and the
isomorphism $ M(A \otimes \mathcal{K}) \cong \mathcal{L}(H)$.
\end{proof}
The second statement of the above lemma can also be deduced from the path connectivity of $U \mathcal{L}(H)$, cf.
\cite[Lemma 1.9]{Phillips89k}.

\begin{proposition} \label{ati44}
Let $X$ be a countably compactly generated space and let $H$ be the
standard Hilbert module over an arbitrary unital $\sigma$-C*-algebra $A$.
Then abelian groups $[X, \mathcal{KC}(H)]$ and ${\rm RK}_{0}(XA)$
are isomorphic.

\end{proposition}
\begin{proof}
The inclusion map
$ \mathcal{KC}(XH) \cap X \mathcal{L}(H) \hookrightarrow \mathcal{KC}(XH)$ induces
the map $$i_{*}:[\mathcal{KC}(XH) \cap X \mathcal{L}(H)] \rightarrow [\mathcal{KC}(XH)].$$
Let $t \in \mathcal{KC}(XH) \cap X \mathcal{L}(H)$ and
$i_{*}([t])=[1]$, then the class of $\rho(t)$ in ${\rm RK}_0(XA)$ is trivial.
Using Lemma \ref{lemma1}, there is a
$k \in \mathcal{K}(XH) \cong X \mathcal{K}(H)$ such that $u=t + k$ is
a unitary element in $ X \mathcal{L}(H)$. Hence, $t$ can be connected to $u$ in
$\mathcal{KC}(XH) \cap X \mathcal{L}(H)$, and since $U X \mathcal{L}(H)$ is path connected,
$u$ can be connected to $1$ in $\mathcal{KC}(XH) \cap X \mathcal{L}(H)$. Consequently, $[t]=[1]$.
As Lemma \ref{ati4}, $[\mathcal{KC}(XH) \cap X \mathcal{L}(H)]$ is an
abelian group and consequently, $i_{*}$ is injective. Surjectivity of $i_{*}$
immediately follows from Lemma \ref{Janich} and the fact that the
Hilbert $XA$-modules $XH$ and $H_{C(X) \otimes A}$ are unitarily equivalent. In view of
Lemma \ref{ati4} and the isomorphism
$\mathcal{K}(XH) \cong X \mathcal{K}(H)$, we have
$$[X, \mathcal{KC}(H)] = [\mathcal{KC}(XH) \cap X \mathcal{L}(H)] \cong [\mathcal{KC}(XH)] \cong {\rm RK}_{0}(XA).$$

\end{proof}

A bounded adjointable operator
$t \in \mathcal{L}(H)$ has polar decomposition if and only if $\overline{t \, H}$ and
$\overline{t^* \, H}$ are orthogonal direct summands. In particular, $t$ has polar decomposition if $t \, H$
is closed.  An operator $t$ has the polar decomposition
$v|t|$ if and only if its adjoint $t^*$ has
the polar decomposition $v^*|t^*|$, see e.g. \cite{FS2, JoitaBook1, SHAMal, WEG}.

\begin{lemma} \label{ati56}
Let $A$ be a unital $\sigma$-C*-algebra and let $f$ be a Fredholm operator on $H=l^2(A)$.
Then there exists a compact perturbation $g$ of $f$
that can be polar decomposed $g=v |g|$ for a partial isometry $v$ with
$p= 1- v^*v \in \mathcal{K}(H)$ and $q=1-vv^* \in \mathcal{K}(H)$.
\end{lemma}
\begin{proof}
Suppose that
$ \rho : \mathcal{L}(H) \to   \mathcal{Q}(H)$ is the quotient map
and $\kappa_n: \mathcal{L}(H) \to \mathcal{L}(H_n)$ is the canonical map. By Lemma \ref{Fred1},
the operator $f$ is Fredholm if and only if $f_n=\kappa_n(f)$ is Fredholm for each $n$.
According to \cite[Proposition 1.7]{MIN},
for every $n \in \mathbb{N}$, there exist
$g_n \in \mathcal{L}(H_n)$, $k_n \in \mathcal{K}(H_n)$ and a partial
isometry $v_n \in \mathcal{L}(H_n)$ such
that $g_n=f_n+k_n$ and $g_n=v_n |g_n|$.
We define the operators $g, v \in \mathcal{L}(H)$ and $k \in \mathcal{K}(H)$ as
the coherent sequences $( g_n )_n$, $( v_n )_n$ and $(k_n)_n$, respectively. Therefore,
$v$ is a partial isometry, $g=f + k$ and $g=v |g|$ by \cite[Proposition 3.3.2]{JoitaBook1}.
Consequently, we have
$\rho(v)= \rho (g) \rho( |g| )^{-1}$ and $\rho(v^*)= \rho (g^*) \rho( |g^*| )^{-1}$
which imply $$\rho(v^* \, v)= \rho (g^*) \rho( |g^*| )^{-1} \  \rho (g) \rho( |g| )^{-1}=
 \rho (g^*) \rho (g) \ \rho( |g| )^{-1} \rho( |g| )^{-1}=1 .$$
Similarly, we get $\rho(v^* \, v)=1$. Then $p= 1- v^*v$ and $q= 1-vv^*$ are compact projections
on $H$, as desired.
\end{proof}

If $\{ \psi_n \} : \{G_n\} \to \{G^{'}_{n} \}$ is a morphism of inverse systems such that each $\psi_n$
is an isomorphism, then the maps
$ \varprojlim_{n} \psi_n : \varprojlim_{n} G_n  \to \varprojlim_{n} G^{'}_{n}$ and
$ \varprojlim^{1}_{n} \psi_n : \varprojlim^{1}_{n} G_n  \to \varprojlim^{1}_{n} G^{'}_{n}$ are isomorphisms
by \cite[Proposition 7.63]{Switzer}. We
use these isomorphisms in the proof of the next result.

\begin{theorem} \label{ati6} Let $X$ be a countably compactly generated space
and let $H$ be the standard Hilbert module over an arbitrary unital $\sigma$-C*-algebra $A$.
The inclusion $i:\mathcal{KC}(H) \hookrightarrow \mathcal{F}(H)$  induces an isomorphism
$i_{*}: [X, \mathcal{KC}(H)] \to [X, \mathcal{F}(H)]$. In particular,
$[X, \mathcal{F}(H)]$ is isomorphic to the abelian group ${\rm RK}_0(XA)$.
\end{theorem}

\begin{proof} We first prove the following assertions for a compact Hausdorff space $X$.
\begin{enumerate}[label=(\subscript{H}{\arabic*})]
\item Any
continuous map $f : X \to \mathcal{F}(H)$ is homotopic to a map whose image is contained
in $\mathcal{KC}(H)$.
\item Any
continuous map $h : [0,1] \times X \to \mathcal{F}(H)$ for which the images of $h_0$ and
$h_1$ are contained in $\mathcal{KC}(H)$, is homotopic to a map whose image is contained
in $\mathcal{KC}(H)$.
\end{enumerate}
To prove the assertion ($H_1$) we suppose $f : X \to \mathcal{F}(H)$ is an arbitrary continues map. It
determines the corresponding bounded adjointable operator $f \in \mathcal{L}(XH)$ by
$(f( \psi))(x)=f(x) \psi(x)$. Since $f \in \mathcal{F}(XH)$,
there exists a compact perturbation of $g \in  \mathcal{L}(XH)$ of $f$ which
satisfies in the fixed conditions of Lemma \ref{ati56}. Via the isomorphism $X \mathcal{K}(H) \cong \mathcal{K}(XH)$
the operator $g-f$ corresponds to a continuous map $k: X \to \mathcal{K}(H)$.
Then for every $\theta \in [0,1]$ we can define $f_{\theta}= f + \theta k : X  \to \mathcal{F}(H)$,
which is a homotopy between the maps $f$ and $g$.

The compact projections $p, q \in \mathcal{K}(XH)$
onto the kernels of $g$ and $g^*$ stratify $qg=0=gp$ and $pg^*=0=g^*q$. The projections
$p$ and $q$ via the isomorphism $X \mathcal{K}(H) \cong \mathcal{K}(XH)$ can be identified
with continuous maps $p, q: X \to \mathcal{K}(H)$, which enable us to define
$a, b : X \to \mathcal{F}(H)$ by
$$a:=g^*g + p ~ \ ~~{\rm and}~ \ ~~ b:=gg^*+q.$$
These maps are continuous and satisfy $bg=ga=gg^*g$ and $ag^*=g^*b=g^*gg^*$. For every
$x \in X$, $a(x)$ and $b(x)$ are positive invertible operators. We therefore have
$b^{ \gamma}g= g a^{ \gamma}$ and $a^{ \gamma}g^*= g^* b^{ \gamma}$, for any real number $ \gamma$.
We define the continuous map
$$w=b^{-1/2}g: X \to \mathcal{KC}(H),$$
where $$1-w^*w=1- a^{-1/2}g^* b^{-1/2}g = a a^{-1}- a^{-1} g^* g=a^{-1}p \in \mathcal{K}(H),$$
$$1-ww^*=1- b^{-1/2}g a^{-1/2}g^* = b b^{-1}- b^{-1} gg^* =b^{-1}q \in \mathcal{K}(H).$$
Then for every $\theta \in [0,1]$ we can define the continuous map $$g_{\theta}=
 (\theta + (1- \theta)b^{-1})^{1/2}g : X  \to \mathcal{F}(H),$$ which is
 a path joining $g$ and $w$. One can easily check that the map $(\theta + (1- \theta)b^{-1})^{1/2}$
 takes its values in the positive invertible operators for all $\theta \in [0,1]$ and so
 the operators $g_{ \theta}(x)$ are Fredholm. To see this suppose that $s_1(x)$  is a left
 generalized inverse for $g(x)$ then $s_1(x) (\theta + (1- \theta)b(x)^{-1})^{-1/2}$ is a left generalized inverse
 of $g_{ \theta}(x)$. If $s_2(x)$ is a right generalized inverse of $g(x)$, then
 $s_2(x) (\theta + (1- \theta)b(x)^{-1})^{-1/2}$ is a right generalized inverse
 of $g_{ \theta}(x)$. Then the continuous map
 $\mathfrak{H}: [0,1] \times X \to \mathcal{F}(H)$ given by
\[  \mathfrak{H}( \theta, x)=
\begin{cases}
    f_{ 2 \theta}(x)    & \text{if } 0 \leq \theta \leq 1/2\\
    g_{ 2- 2 \theta}(x) & \text{if } 1/2 \leq \theta \leq 1
\end{cases}
\]
is a homotopy between the maps $f : X \to \mathcal{F}(H)$ and $\omega : X \to \mathcal{KC}(H)$,
which completes the proof of ($H_1$).
The assertion ($H_2$) can be obtained in a similar way. Surjectivity and
injectivity of the map $i_{*}: [X, \mathcal{KC}(H)] \to [X, \mathcal{F}(H)]$ immediately follow from the
properties ($H_1$) and ($H_2$).

We now suppose $X= \varinjlim_{n} X_n$ when each $X_n$ is a compact Hausdorff space.
Using  \cite[Lemma 2]{Milnor}, we
obtain the following commutative diagram with exact rows
\begin{equation*}
\begin{array}{ccccccccc} &&&  &&  &&  \\
0 & \longrightarrow & \varprojlim_{n} ^{1} [SX_{n}, \, \mathcal{KC}(H)] & \longrightarrow &
[X, \,  \mathcal{KC}(H)] & \longrightarrow & \varprojlim_{n} [X_n, \,  \mathcal{KC}(H)] & \longrightarrow & 0 \\
&& \quad\big\downarrow \cong && \quad\big\downarrow  && \quad\big\downarrow \cong \\
0 & \longrightarrow & \varprojlim_{n} ^{1} [S X_{n}, \, \mathcal{F}(H)] & \longrightarrow & [X, \,  \mathcal{F}(H)]
& \longrightarrow & \varprojlim_{n} [X_n, \,  \mathcal{F}(H)] & \longrightarrow & 0.
\end{array}
\end{equation*}
The vertical maps
$[X_n, \,  \mathcal{KC}(H)] \to [X_n, \,  \mathcal{F}(H)]$ and
$ [S X_{n}, \, \mathcal{KC}(H)] \to [SX_n, \,  \mathcal{F}(H)]$ are
isomorphisms by the first part of the proof. Therefore, the middle
vertical map is an isomorphism by the five lemma. In particular, the abelian
groups $[X, \mathcal{F}(H)]$ and ${\rm RK}_0(XA)$ are isomorphic by Proposition \ref{ati44}.
\end{proof}



Let $A$ be a unital C*-algebra and let $X$ be
a compact Hausdorff space. An $A$-vector bundle over $X$ is a locally trivial bundle over
$X$ whose fibers are finitely generated
projective $A$-modules, with $A$-linear
transition functions. Then, by a result of Rosenberg \cite[Proposition 3.4]{RosenbergBook},
the Grothendieck group of $A$-vector bundles ${\rm K}^{0}(X;A)$ is
naturally isomorphic to ${\rm K}_{0}(XA)$. This fact can also be deduced from
the results of Toritsky and Mingo. However, by \cite[Example 4.9]{Phillips89k},
the Grothendieck groups of $A$-vector bundles over $X$ need not be isomorphic to
${\rm RK}_0(XA)$, when $A$ is a $\sigma$-C*-algebra.

\begin{corollary} \label{ati7} Let $H$ be the
standard Hilbert module over an arbitrary unital $\sigma$-C*-algebra $A$. Then
$\pi_i( \mathcal{F}(H))$ is isomorphic to ${\rm RK}_i( A)$ and satisfies in the
following short exact sequence of abelian groups
\begin{equation}\label{ati71}
  0 \to \varprojlim_{n}\, ^{1} \, \pi_{1-i}(\mathcal{F}(H_n)) \to \pi_i( \mathcal{F}(H) ) \to
  \varprojlim_{n} \pi_i(\mathcal{F}(H_n)) \to 0.
\end{equation}
In particular, we have
$ \pi_i( \mathcal{F}(H) ) \cong \varprojlim_{n} \pi_i(\mathcal{F}(H_n))$ when
the conditions of Proposition \ref{MittLeff} are fulfilled.
\end{corollary}
\begin{proof} The first isomorphism
follows from the following diagram
\begin{equation*}
\begin{array}{ccccccccc} &&&  &&  &&  \\
0 & \longrightarrow & \pi_i( \mathcal{F}(H)) & \longrightarrow &
[S^{i}, \,  \mathcal{F}(H)] & \longrightarrow & [ \mathcal{F}(H)] & \longrightarrow & 0 \\
&& \quad\big\downarrow  && \quad\big\downarrow \cong && \quad\big\downarrow \cong \\
0 & \longrightarrow & {\rm RK}_0( C_0( \mathbb{R}^{i}) \otimes A) & \longrightarrow &  {\rm RK}_0( C(S^{i}) \otimes A)
& \longrightarrow & {\rm RK}_0(A) & \longrightarrow & 0,
\end{array}
\end{equation*}
in which the horizontal maps obtain from inclusion and evaluation at the north pole respectively.
The short exact sequence (\ref{ati71}) is obtained from Lemma \ref{Fred1} and the short exact sequence (\ref{ati1}).
\end{proof}
We close the paper with the following problems which are motivated by Theorem \ref{ati6} and
Corollary \ref{ati7}. Suppose that $H$ is the standard Hilbert module over an arbitrary
unital $\sigma$-C*-algebra $A$.
\begin{problem}\label{prob1}
Characterize those
$\sigma$-C*-algebras $A$ for which
\begin{equation}\label{ati711}
 \cdot \cdot \cdot \to \mathcal{F}(H_{n+1}) \to \mathcal{F}(H_n) \cdot \cdot ~ \cdot ~ \to \mathcal{F}(H_1)
\end{equation}
is a sequence of maps having homotopy lifting property.
\end{problem}

\begin{problem}\label{prob2}
Characterize those
$\sigma$-C*-algebras $A$ for which $\mathcal{KC}(H) \subseteq \mathcal{F}(H)$
is a deformation retract, that is,
there is a map $r: \mathcal{F}(H) \to \mathcal{KC}(H)$ with $r \, i = 1_{ \mathcal{KC}(H) }$
and the map $i \, r$ is homotopy equivalent to $1_{ \mathcal{F}(H) }$.
\end{problem}
If $A$ is a unital C*-algebra and $H=l^2(A)$, one can apply functional calculus to show that
$\mathcal{KC}(H) \subseteq \mathcal{F}(H)$ is a deformation retract.

{\bf Acknowledgement}:\,This research was done while the
author stayed at the Ma\-the\-matisches Institut of the Westf\"alische
Wilhelms-Universit\"at in M\"unster, Germany, in 2015. He would like to
express his thanks to Wend Werner and his
colleagues in functional analysis and operator algebras group for
warm hospitality and great scientific atmosphere.
The author is indebted to
Claude L. Schochet for his helpful comments
on an earlier draft of this paper.
The author also would like to thank
the referee for his/her careful reading and useful comments.

\end{document}